\documentclass[12pt,a4paper]{article}

\usepackage[T1]{fontenc}
\usepackage[utf8]{inputenc}
\usepackage{lmodern}
\usepackage[a4paper,margin=30mm]{geometry}
\usepackage{amsmath,amssymb,amsthm,mathtools,bm}
\usepackage{microtype}
\usepackage{enumitem}
\usepackage{booktabs}
\usepackage{listings}
\usepackage[hidelinks]{hyperref}
\usepackage[nameinlink,noabbrev]{cleveref}
\usepackage{tikz}
\usepackage{tikz-cd}
\usepackage{amscd}

\numberwithin{equation}{section}

\newtheorem{theorem}{Theorem}[section]

\newtheorem{lemma}[theorem]{Lemma}
\newtheorem{corollary}[theorem]{Corollary}
\newtheorem{problem}{Problem}
\theoremstyle{definition}

\newtheorem{remark}[theorem]{Remark}

\newcommand{\Paw}{\mathsf{Paw}}
\newcommand{\Kbar}{\overline{K}}
\newcommand{\Ftwo}{\mathbb{F}_{2}}

\lstdefinestyle{appendixpython}{
  language=Python,
  basicstyle=\ttfamily\scriptsize,
  frame=single,
  framerule=0.4pt,
  breaklines=true,
  columns=fullflexible,
  keepspaces=true,
  showstringspaces=false,
  tabsize=4,
  numbers=none,
  xleftmargin=0pt,
  framexleftmargin=0pt,
  aboveskip=0.5em,
  belowskip=0.5em
}

\renewcommand\refname{References}
\usepackage{authblk}

\title{\bfseries The Distance Spectrum Does Not Determine Bipartiteness}
\author[a]{\textbf{Feifan Gong}}
\author[b]{\textbf{Kehua Wang}}
\author[a]{\textbf{Wei Wang}\thanks{Corresponding author. wang\_weiw@xjtu.edu.cn}}
\affil[a]{School of Mathematics and Statistics, Xi'an Jiaotong University, Xi'an, 710049, P.R. China}
\affil[b]{Department of Computer Science, Columbia University, New York, NY, 10027, USA}
\date{}

\begin{document}

\renewcommand\refname{References}
\maketitle

\begin{abstract}

Over a decade ago, Koolen, Hayat, and Iqbal posed the problem of whether the distance spectrum determines bipartiteness within the class of connected graphs. In this paper, we resolve this problem in the negative: we explicitly construct an infinite family of counterexamples, where each pair comprises a connected bipartite graph and a connected non-bipartite graph with equal distance spectra.\\



\noindent\textbf{Keywords:} Distance matrix; Distance spectrum; Cospectral graphs; Bipartite graph; Graph blow-up

\noindent\textbf{Mathematics Subject Classification:} 05C50, 05C12
\end{abstract}

\section{Introduction}

Throughout the paper, all graphs are finite, simple, and undirected. For a connected graph $X$, define
\[
D(X)=\bigl(d_X(u,v)\bigr)_{u,v\in V(X)},
\qquad
\chi_{D(X)}(\lambda)=\det(\lambda I-D(X)),
\]
where $d_X(u,v)$ is the distance between vertices $u$ and $v$. The eigenvalues (including their multiplicities) of $D(X)$ form the \emph{distance spectrum}; graphs with the same distance spectrum are \emph{distance cospectral}.

The distance spectrum determines several structured graph classes, including complete multipartite graphs and hypercubes~\cite{JinZhang2014,KoolenHayatIqbal2016}; see~\cite{AouchicheHansen2014} for a survey. It does not, however, determine many basic invariants. Distance-cospectral graphs may have different numbers of edges~\cite{Heysse2017}, diameters, or Wiener indices~\cite{AbiadEtAl2017}.

In 2016, Koolen, Hayat, and Iqbal~\cite{KoolenHayatIqbal2016} posed the following problem.

\begin{problem}[Koolen, Hayat, and Iqbal~{\cite[Problem~6.4]{KoolenHayatIqbal2016}}]
\label{prob:bipartiteness}
Does the distance spectrum of a graph determine whether a graph is bipartite?
\end{problem}

Heysse~\cite{Heysse2017} subsequently reported that no distance-cospectral
pair with different bipartiteness exists on at most ten vertices. The purpose of this note is to give an explicit negative answer to Problem~\ref{prob:bipartiteness} by constructing a connected bipartite graph and a connected non-bipartite graph with the same distance spectrum.

Our construction uses nonuniform independent-set blow-ups (see Section~\ref{sec2} for more details). For $X=F[\Kbar_{m_0},\ldots,\Kbar_{m_{r-1}}]$, we prove
\[
D(X)+2I=Z\bigl(D(F)+2I\bigr)Z^{\mathsf T},
\]
where $Z$ is the class-incidence matrix. Consequently, $\chi_{D(X)}$ reduces to an $r\times r$ determinant, so the comparison is exact and takes place at the level of the base graph. Our main contribution is to identify two base graphs and positive integer multiplicities for which the reduced determinants coincide although the resulting graphs have opposite bipartiteness. Applying the reduction to $P_4$ and the paw graph (see Section~\ref{sec3} for more details) gives
\[
G=P_4[\Kbar_1,\Kbar_{16},\Kbar_8,\Kbar_{36}],
\qquad
H=\Paw[\Kbar_4,\Kbar_9,\Kbar_{16},\Kbar_{32}].
\]
These graphs are connected and distance cospectral, while $G$ is bipartite and $H$ is not. Scaling the multiplicities gives an infinite family; the same pairs also differ in size, diameter, Wiener index, and binary distance-matrix rank.

An exact verification program is supplied with the source files. Section~2 gives the blow-up factorization, Sections~3 and~4 prove the counterexample and infinite family, and Appendix~A reproduces the verification program in full.

\section{Independent-set blow-ups}\label{sec2}

Let $F$ be a connected graph with vertex set $\{0,1,\ldots,r-1\}$, where $r\ge 2$, and let
\[
\boldsymbol{m}=(m_0,m_1,\ldots,m_{r-1})\in\mathbb{Z}_{>0}^{r}.
\]
The \emph{independent-set blow-up}
\[
F[\Kbar_{m_0},\Kbar_{m_1},\ldots,\Kbar_{m_{r-1}}]
\]
is obtained by replacing each vertex $i$ with an independent set $V_i$ of size $m_i$ and each edge $ij\in E(F)$ with all edges between $V_i$ and $V_j$. We write $F[\boldsymbol m]$ for this graph.

To describe its distance matrix, put
\[
n=\sum_{i=0}^{r-1}m_i,
\qquad
N=\operatorname{diag}(m_0,m_1,\ldots,m_{r-1}),
\]
and let $Z$ be the $n\times r$ class-incidence matrix defined by
\[
Z_{x,i}=
\begin{cases}
1,&x\in V_i,\\
0,&x\notin V_i.
\end{cases}
\]
Then $Z^{\mathsf T}Z=N$.

The following factorization and determinant reduction are the independent-set specialization of the graph blow-up framework studied by Choudhury and Khare~\cite{ChoudhuryKhare2024}. We include short self-contained proofs in the notation needed here.

\begin{lemma}[Distance factorization]\label{lem:factorization}
Let $X=F[\boldsymbol m]$ and set
\[
W_F=D(F)+2I_r.
\]
Then
\[
D(X)+2I_n=ZW_FZ^{\mathsf T}.
\]
\end{lemma}

\begin{proof}
Suppose first that $x\ne y$ belong to the same class $V_i$. Since $F$ is connected and $r\ge2$, the vertex $i$ has a neighbor in $F$. Thus $x$ and $y$ have a common neighbor in $X$, and hence $d_X(x,y)=2$. If $x\in V_i$ and $y\in V_j$ with $i\ne j$, every path in $X$ projects to a walk in $F$, whereas every shortest $i$--$j$ path in $F$ can be lifted to an $x$--$y$ path in $X$. Consequently,
\[
d_X(x,y)=d_F(i,j).
\]
The matrix identity now follows by comparing entries: both sides equal $2$ on the diagonal and between distinct vertices in the same class, and both equal $d_F(i,j)$ between the classes $V_i$ and $V_j$.
\end{proof}

\begin{lemma}[Reduced distance characteristic polynomial]\label{lem:charpoly}
For $X=F[\boldsymbol m]$,
\[
\chi_{D(X)}(\lambda)
=(\lambda+2)^{n-r}
\det\!\left((\lambda+2)I_r-W_FN\right).
\]
\end{lemma}

\begin{proof}
Put $\mu=\lambda+2$. By Lemma~\ref{lem:factorization},
\[
\chi_{D(X)}(\lambda)
=\det\left(\mu I_n-ZW_FZ^{\mathsf T}\right).
\]
For $\mu\ne0$, Sylvester's determinant identity gives
\begin{align*}
\det\left(\mu I_n-ZW_FZ^{\mathsf T}\right)
&=\mu^n\det\left(I_n-\mu^{-1}ZW_FZ^{\mathsf T}\right)\\
&=\mu^n\det\left(I_r-\mu^{-1}W_FZ^{\mathsf T}Z\right)\\
&=\mu^{n-r}\det\left(\mu I_r-W_FN\right).
\end{align*}
Since both sides are polynomials in $\mu$, the identity holds for every $\mu$.
\end{proof}

\begin{remark}
Although $W_FN$ need not be symmetric, it is similar to the symmetric matrix $N^{1/2}W_FN^{1/2}$. Consequently, all its eigenvalues are real.
\end{remark}

\section{An explicit counterexample}\label{sec3}

Let $P_4$ and $\Paw$ be the path and the paw graph shown in Figure~\ref{fig:bases}. In the formulas below, the vertices of $P_4$ are ordered from left to right, while those of $\Paw$ are ordered as the central, upper, lower, and pendant vertices.

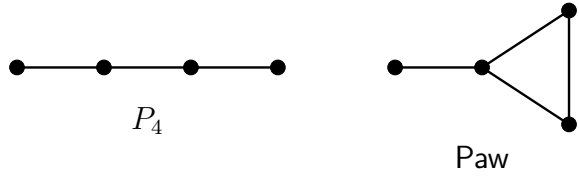
\begin{figure}[!ht]
\centering
\begin{tikzpicture}[
  vertex/.style={circle,fill=black,draw=black,inner sep=0pt,minimum size=5.5pt},
  edge/.style={draw=black,line width=0.9pt}
]
  \coordinate (p0) at (0,0);
  \coordinate (p1) at (1.15,0);
  \coordinate (p2) at (2.30,0);
  \coordinate (p3) at (3.45,0);
  \draw[edge] (p0)--(p1)--(p2)--(p3);
  \foreach \v in {p0,p1,p2,p3} \node[vertex] at (\v) {};
  \node at (1.725,-0.70) {$P_4$};

  \coordinate (a3) at (5.00,0);
  \coordinate (a0) at (6.15,0);
  \coordinate (a1) at (7.30,0.75);
  \coordinate (a2) at (7.30,-0.75);
  \draw[edge] (a3)--(a0)--(a1)--(a2)--(a0);
  \foreach \v in {a0,a1,a2,a3} \node[vertex] at (\v) {};
  \node at (6.15,-1.18) {$\Paw$};
\end{tikzpicture}
\caption{The base graphs $P_4$ and $\Paw$.}
\label{fig:bases}
\end{figure}

For positive variables $b_0,b_1,b_2,b_3$, direct expansion of the determinant gives
\begin{align}
&\det\!\left(
\mu I_4-\bigl(D(P_4)+2I_4\bigr)
\operatorname{diag}(b_0,b_1,b_2,b_3)
\right)\notag\\
&\quad=
\mu^4
-2(b_0+b_1+b_2+b_3)\mu^3\notag\\
&\qquad+
\left(3b_0b_1-5b_0b_3+3b_1b_2+3b_2b_3\right)\mu^2\notag\\
&\qquad+8b_0b_3(b_1+b_2)\mu
-12b_0b_1b_2b_3.
\label{eq:path-polynomial}
\end{align}
Likewise, for positive variables $a_0,a_1,a_2,a_3$,
\begin{align}
&\det\!\left(
\mu I_4-\bigl(D(\Paw)+2I_4\bigr)
\operatorname{diag}(a_0,a_1,a_2,a_3)
\right)\notag\\
&\quad=
\mu^4
-2(a_0+a_1+a_2+a_3)\mu^3\notag\\
&\qquad+3(a_0a_1+a_0a_2+a_0a_3+a_1a_2)\mu^2\notag\\
&\qquad+2a_1a_2(a_3-2a_0)\mu
-3a_0a_1a_2a_3.
\label{eq:paw-polynomial}
\end{align}

Since all parameters are positive, the constant terms in \eqref{eq:path-polynomial} and \eqref{eq:paw-polynomial} are nonzero. Hence Lemma~\ref{lem:charpoly} shows that equality of the full distance characteristic polynomials first forces equality of the powers of $\mu$, and therefore equality of the graph orders. Once the orders agree, equality is equivalent to coefficientwise equality of the two reduced quartics. Thus the two blow-ups have the same distance characteristic polynomial precisely when the positive integer parameters satisfy the following coefficient-matching system:
\begin{equation}
\left\{
\begin{aligned}
b_0+b_1+b_2+b_3
&=a_0+a_1+a_2+a_3,\\
3b_0b_1-5b_0b_3+3b_1b_2+3b_2b_3
&=3(a_0a_1+a_0a_2+a_0a_3+a_1a_2),\\
8b_0b_3(b_1+b_2)
&=2a_1a_2(a_3-2a_0),\\
4b_0b_1b_2b_3
&=a_0a_1a_2a_3.
\end{aligned}
\right.
\label{eq:coefficient-matching}
\end{equation}
The first equation matches both the graph orders and the coefficients of $\mu^3$, while the remaining equations match the coefficients of $\mu^2$, $\mu$, and the constant term, respectively. The parameter tuples used below form a positive integer solution of \eqref{eq:coefficient-matching} with common sum $61$.

We now state the promised counterexample.

\begin{theorem}\label{thm:counterexample}
Define
\[
G=P_4[\Kbar_{1},\Kbar_{16},\Kbar_{8},\Kbar_{36}]
\]
and
\[
H=\Paw[\Kbar_{4},\Kbar_{9},\Kbar_{16},\Kbar_{32}].
\]
Then $G$ and $H$ are connected and distance cospectral. Moreover, $G$ is bipartite and $H$ is non-bipartite.
\end{theorem}

\begin{proof}
Since every blow-up class is nonempty and both base graphs are connected, the graphs $G$ and $H$ are connected.

The path $P_4$ has bipartition $\{0,2\}\sqcup\{1,3\}$. Therefore,
\[
(V_0\cup V_2)\sqcup(V_1\cup V_3)
\]
is a bipartition of $G$, with part sizes $9$ and $52$. Hence $G$ is bipartite.

Choose one vertex from each of $V_0,V_1,V_2$ in $H$. These three vertices induce a triangle because $0,1,2$ form a triangle in $\Paw$. Hence $H$ is non-bipartite.

It remains to establish distance cospectrality. The distance matrices of the base graphs are
\[
D(P_4)=
\begin{pmatrix}
0&1&2&3\\
1&0&1&2\\
2&1&0&1\\
3&2&1&0
\end{pmatrix},
\qquad
D(\Paw)=
\begin{pmatrix}
0&1&1&1\\
1&0&1&2\\
1&1&0&2\\
1&2&2&0
\end{pmatrix}.
\]
Consequently, the reduced matrices in Lemma~\ref{lem:charpoly} are
\[
Q_G=\bigl(D(P_4)+2I_4\bigr)\operatorname{diag}(1,16,8,36)
=
\begin{pmatrix}
2&16&16&108\\
1&32&8&72\\
2&16&16&36\\
3&32&8&72
\end{pmatrix}
\]
and
\[
Q_H=\bigl(D(\Paw)+2I_4\bigr)\operatorname{diag}(4,9,16,32)
=
\begin{pmatrix}
8&9&16&32\\
4&18&16&64\\
4&9&32&64\\
4&18&32&64
\end{pmatrix}.
\]
Substituting
\[
(b_0,b_1,b_2,b_3)=(1,16,8,36)
\]
into \eqref{eq:path-polynomial}, and
\[
(a_0,a_1,a_2,a_3)=(4,9,16,32)
\]
into \eqref{eq:paw-polynomial}, we obtain in both cases
\[
\det(\mu I_4-Q_G)
=
\det(\mu I_4-Q_H)
=
\mu^4-122\mu^3+1116\mu^2+6912\mu-55296.
\]
Both graphs have order $61$. Applying Lemma~\ref{lem:charpoly} with $r=4$ and $\mu=\lambda+2$ gives
\begin{align*}
\chi_{D(G)}(\lambda)
=\chi_{D(H)}(\lambda)
&=(\lambda+2)^{57}
\Bigl((\lambda+2)^4-122(\lambda+2)^3\\
&\hspace{24mm}+1116(\lambda+2)^2
+6912(\lambda+2)-55296\Bigr)\\
&=(\lambda+2)^{57}
\bigl(\lambda^4-114\lambda^3+408\lambda^2\\
&\hspace{35mm}+9944\lambda-37968\bigr).
\end{align*}
Therefore, $G$ and $H$ are distance cospectral.
\end{proof}

For an independent exact verification, the program reproduced in Appendix~A constructs the two full graphs of order $61$ and checks their connectivity, bipartiteness, and equality of their distance characteristic polynomials.

\begin{corollary}\label{cor:further-invariants}
The graphs $G$ and $H$ in Theorem~\ref{thm:counterexample} have different numbers of edges, diameters, Wiener indices, and ranks of their distance matrices modulo $2$. More precisely,
\begin{center}
\small
\renewcommand{\arraystretch}{1.2}
\begin{tabular}{@{}lcccc@{}}
\toprule
Graph & $|E|$ & $\operatorname{diam}$ & $W$ & $\operatorname{rank}_{\Ftwo}(D\bmod2)$\\
\midrule
$G$ & $432$ & $3$ & $3264$ & $2$\\
$H$ & $372$ & $2$ & $3288$ & $4$\\
\bottomrule
\end{tabular}
\end{center}
In particular, the rank of the distance matrix modulo $2$ is not determined by the real distance spectrum.
\end{corollary}

\begin{proof}
For an independent-set blow-up $X=F[\boldsymbol m]$, direct counting gives
\[
|E(X)|=\sum_{ij\in E(F)}m_im_j
\]
and
\[
W(X)=2\sum_i\binom{m_i}{2}
+\sum_{i<j}m_im_jd_F(i,j),
\]
where $W(X)$ denotes the Wiener index. Hence
\begin{align*}
|E(G)|&=1\cdot16+16\cdot8+8\cdot36=432,\\
|E(H)|&=4\cdot9+9\cdot16+16\cdot4+4\cdot32=372,
\end{align*}
and
\begin{align*}
W(G)
&=2(0+120+28+630)\\
&\quad +(16+16+108+128+1152+288)=3264,\\
W(H)
&=2(6+36+120+496)\\
&\quad +(36+64+128+144+576+1024)=3288.
\end{align*}
The distance description in Lemma~\ref{lem:factorization} also gives
$\operatorname{diam}(G)=3$ and $\operatorname{diam}(H)=2$.

Reducing Lemma~\ref{lem:factorization} modulo $2$ yields
\[
D(X)\equiv ZD(F)Z^{\mathsf T}\pmod2.
\]
This factorization gives
\(\operatorname{rank}_{\Ftwo}D(X)
\le \operatorname{rank}_{\Ftwo}D(F)\).
Conversely, choosing one representative from each blow-up class produces
$D(F)\bmod2$ as a principal submatrix of $D(X)\bmod2$. Therefore,
\[
\operatorname{rank}_{\Ftwo}(D(X)\bmod2)
=\operatorname{rank}_{\Ftwo}(D(F)\bmod2).
\]
For the two base graphs, the relevant matrices are
\[
D(P_4)\bmod2=
\begin{pmatrix}
0&1&0&1\\
1&0&1&0\\
0&1&0&1\\
1&0&1&0
\end{pmatrix},
\qquad
D(\Paw)\bmod2=
\begin{pmatrix}
0&1&1&1\\
1&0&1&0\\
1&1&0&0\\
1&0&0&0
\end{pmatrix}.
\]
The first matrix has rank $2$ over $\Ftwo$, whereas the second has determinant $1$ over $\Ftwo$ and hence rank $4$. This proves all the assertions.
\end{proof}

\section{An infinite family}

\begin{corollary}\label{cor:family}
For every positive integer $t$, define
\[
G_t=P_4[\Kbar_{t},\Kbar_{16t},\Kbar_{8t},\Kbar_{36t}]
\]
and
\[
H_t=\Paw[\Kbar_{4t},\Kbar_{9t},\Kbar_{16t},\Kbar_{32t}].
\]
Then $G_t$ and $H_t$ are connected and distance cospectral of order $61t$. Moreover, $G_t$ is bipartite, whereas $H_t$ is non-bipartite. Their common distance characteristic polynomial is
\begin{align*}
&(\lambda+2)^{61t-4}
\Bigl((\lambda+2)^4
-122t(\lambda+2)^3
+1116t^2(\lambda+2)^2\\
&\hspace{45mm}
+6912t^3(\lambda+2)-55296t^4\Bigr).
\end{align*}
\end{corollary}

\begin{proof}
The connectivity and bipartiteness assertions follow exactly as in Theorem~\ref{thm:counterexample}. Scaling every blow-up multiplicity by $t$ replaces the reduced matrices $Q_G$ and $Q_H$ by $tQ_G$ and $tQ_H$, respectively. Therefore,
\begin{align*}
\det(\mu I_4-tQ_G)
&=\det(\mu I_4-tQ_H)\\
&=\mu^4-122t\mu^3+1116t^2\mu^2+6912t^3\mu-55296t^4.
\end{align*}
The result now follows from Lemma~\ref{lem:charpoly}.
\end{proof}

The scaled pairs retain the further structural differences exhibited by $G$ and $H$. Indeed, the edge counts scale quadratically, and Lemma~\ref{lem:factorization} shows that the diameters remain unchanged. Applying the Wiener-index formula above to the scaled multiplicities gives
\[
W(G_t)=3325t^2-61t,
\qquad
W(H_t)=3349t^2-61t.
\]
The binary rank identity in Section~3 depends only on the base graph and therefore holds for every positive integer $t$. These invariants are summarized below.

\begin{center}
\centering
\small
\renewcommand{\arraystretch}{1.2}
\begin{tabular}{@{}lcccc@{}}
\toprule
Graph & $|E|$ & $\operatorname{diam}$ & $W$ & $\operatorname{rank}_{\Ftwo}(D\bmod2)$\\
\midrule
$G_t$ & $432t^2$ & $3$ & $3325t^2-61t$ & $2$\\
$H_t$ & $372t^2$ & $2$ & $3349t^2-61t$ & $4$\\
\bottomrule
\end{tabular}
\end{center}

\section{Conclusion}

This paper gives an explicit negative solution to Problem~\ref{prob:bipartiteness}: the graphs in Theorem~\ref{thm:counterexample} are connected and distance cospectral, but have opposite bipartiteness. The independent-set blow-up factorization reduces the spectral comparison to a $4\times4$ determinant, giving an exact proof and, under uniform scaling, infinitely many such pairs. The construction also separates the number of edges, diameter, Wiener index, and binary distance-matrix rank. Two natural follow-up problems are to determine the minimum possible
order of such a pair and to classify all positive integer solutions of the
coefficient-matching equations \eqref{eq:coefficient-matching}.

\section*{Acknowledgements}

The Python codes provided in Appendix A was generated with assistance from \mbox{GPT-5.6-sol}. All computational outputs have been independently verified by the authors. The authors bear full responsibility for the accuracy of the manuscript.

\clearpage
\appendix
\renewcommand{\thesection}{Appendix~\Alph{section}}

\section{Exact computational verification}

The following standalone Python program takes the edge lists of $P_4$ and $\Paw$ together with their blow-up sizes, constructs the two full graphs, and compares the characteristic polynomials of their distance matrices exactly.

\begin{lstlisting}[style=appendixpython]
import networkx as nx
import sympy as sp

E_P4 = [(0, 1), (1, 2), (2, 3)]
M_P4 = [1, 16, 8, 36]
E_PAW = [(0, 1), (0, 2), (0, 3), (1, 2)]
M_PAW = [4, 9, 16, 32]


def blow_up(edges, sizes):
    cells = []
    first = 0
    for size in sizes:
        cells.append(list(range(first, first + size)))
        first += size

    graph = nx.Graph()
    graph.add_nodes_from(range(first))
    for i, j in edges:
        for u in cells[i]:
            for v in cells[j]:
                graph.add_edge(u, v)
    return graph


def distance_matrix(graph):
    n = graph.number_of_nodes()
    distances = dict(nx.all_pairs_shortest_path_length(graph))
    return sp.Matrix([
        [distances[i][j] for j in range(n)]
        for i in range(n)
    ])


G = blow_up(E_P4, M_P4)
H = blow_up(E_PAW, M_PAW)
DG = distance_matrix(G)
DH = distance_matrix(H)
x = sp.symbols("x")

assert G.number_of_nodes() == H.number_of_nodes() == 61
assert nx.is_connected(G) and nx.is_connected(H)
assert nx.is_bipartite(G) and not nx.is_bipartite(H)
assert G.number_of_edges() == 432 and H.number_of_edges() == 372
assert nx.diameter(G) == 3 and nx.diameter(H) == 2

chi_G = DG.charpoly(x).as_expr()
chi_H = DH.charpoly(x).as_expr()
assert chi_G == chi_H

expected = (x + 2)**57 * (
    x**4 - 114*x**3 + 408*x**2 + 9944*x - 37968
)
assert sp.expand(chi_G - expected) == 0

print(sp.factor(chi_G))
\end{lstlisting}

\noindent Running the program gives
\begin{lstlisting}[style=appendixpython]
(x + 2)**57*(x**4 - 114*x**3 + 408*x**2 + 9944*x - 37968)
\end{lstlisting}

\end{document}